\documentclass[onefignum,onetabnum]{siamart171218}

\usepackage{lipsum}
\usepackage{amsfonts}
\usepackage{graphicx}
\usepackage{epstopdf}
\usepackage{algorithmic}
\ifpdf
  \DeclareGraphicsExtensions{.eps,.pdf,.png,.jpg}
\else
  \DeclareGraphicsExtensions{.eps}
\fi

\newsiamremark{remark}{Remark}
\newsiamremark{hypothesis}{Hypothesis}
\crefname{hypothesis}{Hypothesis}{Hypotheses}
\newsiamthm{claim}{Claim}

\headers{$D-$Magic and Antimagic Labelings of Hypercubes}{P. Anuwiksa, A. Munemasa, and R. Simanjuntak}

\title{$D-$Magic and Antimagic Labelings of Hypercubes\thanks{Submitted to the editors DATE.
\funding{This work was funded by}}}

\author{Palton Anuwiksa\thanks{Master Program in Mathematics, Institut Teknologi Bandung, Indonesia 
  (\email{anuwiksapalton@gmail.com}).} 
\and Akihiro Munemasa\thanks{Graduate School of Information Science, Tohoku University, Japan 
  (\email{munemasa@math.is.tohoku.ac.jp}).} 
\and Rinovia Simanjuntak\thanks{Combinatorial Mathematics Group, Institut Teknologi Bandung, Indonesia 
  (\email{rino@math.itb.ac.id}).}}

\usepackage{amsopn}

\makeatletter
\newcommand*{\addFileDependency}[1]{
  \typeout{(#1)}
  \@addtofilelist{#1}
  \IfFileExists{#1}{}{\typeout{No file #1.}}
}
\makeatother

\newcommand*{\myexternaldocument}[1]{%
    \externaldocument{#1}%
    \addFileDependency{#1.tex}%
    \addFileDependency{#1.aux}%
}

\ifpdf
\hypersetup{
  pdftitle={D-Magic and Antimagic Labelings of Hypercubes},
  pdfauthor={P. Anuwiksa, A. Munemasa, and R. Simanjuntak}}
\fi

\myexternaldocument{ex_supplement}

\newtheorem{observation}{Observation}
\newtheorem{conjecture}{Conjecture}
\newtheorem{openproblem}[theorem]{Open Problem}

\DeclareMathOperator{\kernel}{Ker}
\newcommand{\allone}{\mathbf{1}}

\newcommand{\F}{\mathbb{F}}
\newcommand{\Z}{\mathbb{Z}}
\newcommand{\bs}{\boldsymbol{s}}
\newcommand{\bk}{\boldsymbol{k}}
\newcommand{\bu}{\boldsymbol{u}}
\newcommand{\bv}{\boldsymbol{v}}
\newcommand{\ba}{\boldsymbol{a}}
\newcommand{\nexteqv}{\displaybreak[0]\\ &\iff}

\newcommand{\nexteq}{\displaybreak[0]\\ &=}
\newcommand{\refby}[1]{&&\text{(by (\ref{#1}))}}

\begin{document}

\maketitle

\begin{abstract}
For a set of distances $D$, a graph $G$ of order $n$ is said to be \emph{$D-$magic} if there exists a bijection $f:V\rightarrow \{1,2, \ldots, n\}$ and a constant $k$ such that for any vertex $x$, $\sum_{y\in N_D(x)} f(y) =k$, where $N_D(x)=\{y|d(y,x)=j, j\in D\}$. 	

In this paper we shall find sets of distances $D$s, such that the hypercube is $D-$magic. We shall utilise well-known properties of (bipartite) distance-regular graphs to construct the $D-$magic labelings.
\end{abstract}

\begin{keywords}
  $D$-magic labeling, $D$-antimagic labeling, distance-regular graph, hypercube
\end{keywords}

\begin{AMS}
  05C12, 05C78
\end{AMS}

\section{Introduction}
We denote by $G=(V,E)$ a finite undirected simple graph of order $n$ and diameter $d$. For an integer $i$ ($0\leq i\leq d$), we define $A_i$ as the $i-$distance matrix of $G$. When $i=1$, the matrix $A_1$ is the adjacency matrix of $G$, and sometime it is denoted by simply $A$. For an integer $i$ ($0\leq i\leq d$) and a vertex $x$ in $G$, we define $G_i(x)$ as the set of all vertices at distance $i$ from $x$. The \textit{open neighborhood} of $x$ is $N(x)=G_1(x)$ and the \textit{closed neighborhood} of $x$ is $N[x]=G_0(x)\bigcup G_1(x)$. In general, for a set of distances $D\subseteq\{0,1, \ldots, d\}$, the \textit{$D-$neighborhood} of a vertex $x$ is $N_D(x)=\bigcup_{i\in D} G_i$. For other standard graph theoretic notations and definitions we refer to Diestel \cite{Diest}.

Magic squares are among the more popular mathematical recreations and in the early 1960s, Sedl\'{a}\v{c}ek  \cite{ref_Sed} asked whether the "magic" ideas could be applied to graph. He introduced a graph labeling where the edges of a graph are labeled with distinct real numbers such that the sum of edge-labels incident with each vertex equal to a constant, independent of the choice of vertex. It is obvious that the complete bipartite graph $K_{n,n}$ can be labeled by elements of a magic square of size $n$. This labeling was called \emph{magic labeling}, but then it becomes known as the \emph{vertex-magic edge labeling}. Since then, many variations of magic labelings have been defined, and the most recent was introduced by O'Neal and Slater in 2013 \cite{OS13}.

\begin{definition}
For a graph $G$ and a set of distances $D$, a bijection $f:V(G)\rightarrow \{1,2,\ldots,n\}$ is called a \emph{$D-$magic labeling} of $G$ if there exists a constant $k$ called the \emph{magic constant} such that for any vertex $x$, \emph{the weight of $x$}, $w(x)=\sum_{y\in N_D(x)} f(y) = k$. 

In the case that the weight $w(x)$ is distinct for every vertex $x$, $f$ is called	a \emph{$D-$antimagic labeling} of $G$. In particular, if the set of weights $\{w(x)\mid x\in V(G)\}$ constitutes an arithmetic progression starting at $\alpha$ with difference $\delta > 0$, then $f$ is called an \emph{$(\alpha,\delta)-D-$antimagic labeling} of $G$. 

Any graph admitting a $D-$magic (resp.\ $D-$antimagic, $(\alpha,\delta)-D-$antimagic) labeling is called a \emph{$D-$magic (resp.\ $D-$antimagic, $(\alpha,\delta)-D-$antimagic) graph}.

If $D=\{1\}$, a $D-$magic (resp.\ $D-$antimagic, $(\alpha,\delta)-D-$antimagic) labeling is known as a \emph{distance magic (resp.\ distance antimagic, $(\alpha,\delta)-$distance antimagic) labeling}, and if $D=\{0,1\}$, a $D-$magic (resp.\ $D-$antimagic, $(\alpha,\delta)-D-$antimagic) labeling is called a \emph{closed distance magic (resp.\ closed distance antimagic, $(\alpha,\delta)-$closed distance antimagic) labeling}.
\end{definition}

The following observation is a direct consequence of the $D-$magic labeling definition.
\begin{observation} \label{union}
Let $D_1$ and $D_2$ be two disjoint sets of distances. If a graph is both $D_1-$magic and $D_2-$magic then it is also $(D_1 \cup D_2)-$magic.
\end{observation} 

In \cite{MRS03}, it was proved that if $G$ is an $r$-regular graph with order $n$ admitting a distance magic labeling, then the magic constant is $\frac{r(n+1)}{2}$. Using this result and the Handshaking Lemma, it is clear that every regular graph with odd degree is not distance magic. In 2004, Acharya, Rao, Singh, and Parameswaran conjectured that distance magic labelings do not exist for all hypercubes of order at least 4, including the ones with even degrees.
\begin{conjecture}\cite{RSP}
\label{conj:hypercube}
For any even integer $n \geq 4$, the $n$-dimensional hypercube $Q_n$ is not distance magic.
\end{conjecture}
The conjecture was proved to be true in \cite{CFKR16} for $n\equiv 0\pmod 4$. However, positive answers were instead obtained for $n\equiv 2\pmod 4$ \cite{GK13}; which lead us to the following theorem.
\begin{theorem}[{\cite{CFKR16,GK13}}] \label{QnDM}
$Q_n$ has distance magic labeling if and only if $n\equiv 2\pmod4$.
\end{theorem}

For other set of distances $D$, there is only one known $D-$magic labelings as follow.
\begin{theorem}[{\cite{GK13}}] \label{j}
For every $n\equiv 2\pmod4$, there exist a $\{j\}$-distance magic labeling of the hypercube $Q_n$ for every odd $j$, $1\leq j\leq n$.
\end{theorem}

In this paper we shall find sets of distances $D$s for which $D-$magic labelings for hypercubes exist (Section \ref{sec:3}). In order to do so, we shall develop more general results for distance-regular graphs, in particular for bipartite distance-regular graphs which are antipodal double covers (Section \ref{sec:2}).

\section{$D-$Magic Labelings for Distance-Regular Graphs}
\label{sec:2}

We shall use the definition of a distance-regular graph as stated in \cite{Brouwer1989}.
\begin{definition}
A connected graph $G$ of diameter $d$ is called \emph{distance-regular} if there are non-negative integers $b_i, c_i (0\leq i \leq d)$ such that for any two vertices $x$ and $y$ in $G$ at distance $i=d(x,y)$, there are precisely $c_i$ neighbours of $y$ in $G_{i-1}(x)$ and $b_i$ neighbours of $y$ in $G_{i+1}(x)$. Clearly, $G$ is regular with degree $r=b_0$, $b_d=c_0=0$, and $c_1=1$. The sequence $i(G)=\{r,b_1,\ldots,b_{d-1};1,c_2,\ldots,c_d\}$ is called the \emph{intersection array} of $G$; and the numbers $a_i, b_i$, and $c_i (i=0,\ldots,d)$, where $a_i=r-b_i-c_i$, are called the \emph{intersection numbers} of $G$.
\end{definition}

The spectrum of a distance-regular graph can be searched by considering a tridiagonal $(d+1)\times (d+1)$ matrix $B$ as stated in the following.
\begin{theorem} \cite{Bi96} \label{B}
If $G$ is a distance-regular graph of diameter $d$ and intersection array $\{r,b_1,\ldots,b_{d-1};1,c_2,\ldots,c_d\}$, then $G$ has $d+1$ distinct eigenvalues which are the eigenvalues of the tridiagonal $(d+1)\times(d+1)-$matrix
\[B=\left(
  \begin{array}{cccccc}
    0 & 1   & 0       & \cdots  &        & 0   \\
    r & a_1 & c_2     &         &        &     \\
    0 & b_1 & a_2     & c_3     &        &     \\
      &     & b_2     & a_3     & \ddots &     \\
      &     &         & \ddots  & \ddots & c_d \\
    0 &     & \cdots  &         & b_{d-1}& a_d \\
  \end{array}
\right).\]
\end{theorem}
In subsequent theorems and proofs, we refer to the matrix $B$ of $G$ as the tridiagonal matrix in Theorem \ref{B}. 

For a vertex $x$ and a labeling of vertices $l$, we denote by $S_i(x) = \sum_{y\in G_i(x)}l(y)$, the sum of labels of all vertices in $G_i(x)$. It is clear that $S_0(x)=l(x)$. We define two vectors of $x$ as follow: the vector $\bs(x)$ as $(S_i(x))_{i=0}^d$ and the vector $\bk$ as $(|G_i(x)|)_{i=0}^d$. Now we are ready to provide necessary conditions for the existence of distance and closed distance magic labelings for distance-regular graphs.
\begin{lemma}\label{lem:KerB0}
Let $G$ be a distance-regular graph of diameter $d$, $x$ be a vertex in $G$, and $l$ be a labeling of vertices in $G$.
If $l$ is a distance magic labeling with magic constant $k'$, then
\[B\bs(x)=k'\bk.\]
If $l$ is a closed distance magic labeling with magic constant $k'$, then
\[(I+B)\bs(x)=k'\bk.\]
\end{lemma}
\begin{proof}
Suppose that $l$ is a distance magic labeling of $G$ with magic constant $k'$. For a vertex $x$ in $G$ and $1\leq i\leq d-1$, consider the sum of weights of vertices in $G_i(x)$,
\[\sum_{y \in G_i(x)}w(y)=\sum_{y\in G_i(x)} \sum_{z\in N(y)}l(z).\]
In this equation, the label of every vertex in $G_{i-1}(x)$ appears $b_{i-1}$ times, the label of every vertex in $G_i(x)$ appears $a_i$ times, and the label of every vertex in $G_{i+1}(x)$ appears $c_{i+1}$ times. Thus, the following holds.
\[k'k_i = S_{i-1}(x) b_{i-1} + S_i(x) a_i + S_{i+1}(x) c_{i+1}.\]
This proves that $k'\bk=B\bs(x)$. The second statement can be proved in a similar manner.
\end{proof}

A distance-regular graph $G$ of diameter $d$ is called an \emph{antipodal double cover} if $|G_d(x)|=1$ for some (and hence all) vertex $x$ in $G$. The unique vertex in $G_d(x)$ is called the \emph{antipode} of $x$ and denoted by $x'$ in what follows.
\begin{lemma}\label{lem:KerB1}
Let $G$ be a bipartite distance-regular graph which is an antipodal double cover.
If $\kernel B\neq 0$, then the diameter $d$ of $G$ is even. Moreover, in this case, $\kernel B$ has a basis of the form
\begin{equation}
\begin{pmatrix}1\\ 0 \\ \vdots \\ (-1)^{d/2}\end{pmatrix}.
\end{equation}
\end{lemma}
\begin{proof}
The tridiagonal matrix $B$ of $G$ as in Theorem \ref{B} is of the form
\begin{equation}
B= \begin{pmatrix}
	0 & 1      &           &       &       &  0\\
	r & 0      & c_2       &       &       &   \\
	  & c_{d-1}& 0         & c_3   &       &   \\
	  &        & c_{d-2}   & 0     &\ddots &   \\
	  &        &           & \ddots&\ddots &  r\\
	0 &        &           &       & 1     &  0
\end{pmatrix}.
\end{equation}
Let $\bu=(u_i)_{i=0}^d$ be an eigenvector of $B$, normalized to have $u_0=1$. Then $u_1=0$, and $u_i$ can be recursively determined by the condition $B\bu=0$, and we obtain $u_d=(-1)^{d/2}$.
\end{proof}

\begin{lemma}\label{lem:KerB}
Let $G$ be a distance-regular graph which is an antipodal double cover. If $G$ admits a distance (resp.\ closed) magic labeling, then $\kernel B$ (resp. $\kernel(I+B)$) has a basis of the form
\begin{equation}
\begin{pmatrix}1\\ \vdots \\ -1\end{pmatrix}.
\end{equation}
\end{lemma}
\begin{proof}
Suppose that $l$ is a distance magic labeling of $G$ with magic constant $k'$.
For a vertex $x$ in $G$, we have
\[B\bs(x)=B\bs(x')=k'\bk=\frac{k'}{r}B\bk,\]
where $r$ is the degree of $G$.
This implies $\bs(x)-\bs(x')\in \kernel B$ and gives the desired description of $\kernel B$. If $G$ admits a closed distance magic labeling $l$, then an analogous argument shows $\bs(x)-\bs(x')\in\kernel(I+B)$.
\end{proof}

The following theorem enables us to construct new $D-$magic or antimagic labelings for a (bipartite) distance-regular graph, provided that the graph is distance magic.
\begin{theorem} \label{BDRG}
Let $G$ be a distance-regular graph of diameter $d$ admitting a distance magic labeling $l$.
If $D$ is a non-empty subset of $\subseteq\{0,1,\dots,n\}$ then $l$ is either $D$-magic labeling or $(\alpha,\delta)-D-$antimagic
labeling for some $\alpha,\delta$. Moreover, if $G$ is bipartite, then $l$ is
$D-$magic for all non-empty $D\subseteq\{1,3,5,\dots\}$.
\end{theorem}
\begin{proof}
Regard the distance magic labeling $l$ as a vector indexed by $V(G)$. Then $Al\in\langle\allone\rangle$. Since $G$ is distance-regular, there exists a polynomial $p$ such that $p(A)=\sum_{i\in D}A_i$. If $p$ has constant term $d$, then
$p(A)l\in dl+\langle\allone\rangle$.
Thus, $l$ is a $D$-magic labeling or $(a,d)$-$D$-antimagic labeling for some $a$, according to $d=0$ or not.

Assume now that $G$ is bipartite, and that $D$ consists of odd positive integers. Then $\sum_{i\in D}A_i$ is an odd polynomial in $A$, and thus $\sum_{i\in D}A_il\in\langle\allone\rangle$.
\end{proof}

Note that Theorem \ref{j} is a direct consequence of Theorems \ref{QnDM} and \ref{BDRG}.

Now we are ready to prove the following theorem, which provides us many new $D-$magic labelings for a distance-regular graph which is an antipodal double cover, provided that the graph is either distance or closed distance magic. Additionally, the theorem also provides a necessary condition for the existence of a distance magic labeling for such a graph.
\begin{theorem} \label{n=2mod4}
Let $G$ be a distance-regular graph which is an antipodal double-cover.
If $l$ is a distance magic labeling or a closed distance magic of $G$ then
$l$ is a $\{j,d-j\}-$magic labeling for every $j$. Moreover, if $G$ is bipartite and
$l$ is a distance magic labeling, then
$d\equiv 2\pmod{4}$.
\end{theorem}
\begin{proof}
From Lemma \ref{lem:KerB}, $\kernel B$ has a basis $\bu$ of the form $u_0=1=-u_d$. Then $l(x)+l(x')$ is a constant independent of $x$. Consequently, for every $x\in V(G)$,
\begin{align*}
S_j(x)+S_{d-j}(x)&=\sum_{y\in G_j(x)}(l(y)+l(y'))
\nexteq
\sum_{y\in G_j(x)}(|V(G)|+1)
\nexteq
(|V(G)|+1)|G_j (x)|.
\end{align*}
Therefore, $G$ is $\{j,n-j\}$-magic.

If $G$ is bipartite, then $\kernel B\neq0$ implies that $d$ must be even. Recursively comparing entries of
$B\bs(x)=k'\bk$,
there exist constants $a$ and $b$ such that
$l(x') = a + b l(x) $. More explicitly, $b=(-1)^{d/2}$.
Switching the role of $x$ and $x'$, we find $l(x) = a + b l(x') $.
This forces $b=-1$, and hence $d\equiv2\pmod4$.
\end{proof}

Theorem \ref{n=2mod4} provides an alternative proof for the non-existence of distance magic labelings for $n$-dimensional hypercube $Q_n$, when $n\not\equiv 2\pmod 4$, which was originally proved in \cite{CFKR16}. We could also use the theorem to prove the following result.
\begin{corollary}
Hadamard graphs are not distance magic.
\end{corollary}
\begin{proof} It is well known that a Hadamard graph is a bipartite distance-regular graph which is an antipodal double cover with diameter $4$.
\end{proof}

\section{$D-$Magic Labelings for Hypercubes}
\label{sec:3}

Recall that the $n-$hypercube $Q_n$ is a bipartite distance-regular graph which is an antipodal double cover with diameter $n$. As direct consequences of Observation \ref{union}, Theorems \ref{QnDM}, \ref{j}, \ref{BDRG} and \ref{n=2mod4}, we obtain the following sets of distances in which $D-$magic labelings exist for the hypercube $Q_n$, where $n\equiv 2\pmod4$.

\begin{theorem} \label{2mod4}
If $n\equiv 2\pmod4$ then there exists a $D$-magic labeling of $Q_n$ whenever $D$ is of the form
\begin{equation*}
E\cup\bigcup_{i\in I}\{i,n-i\},
\end{equation*}
where $E$ is a non-empty subset of $\{1,3,\ldots, n-1\}$, $I\subseteq\{0,1,\ldots,n/2\}$, and
\[E\cap\{i,n-i\}=\emptyset\quad(i\in I).\]
\end{theorem}

Thus a natural question would be:
\begin{openproblem} \label{D_2mod}
"Are the sets $D$s in Theorem \ref{2mod4} the only $D$ for which $Q_n$, $n\equiv 2\pmod4$, is $D$-magic?"
\end{openproblem} 
If the answer of the question in Open Problem \ref{D_2mod} is positive, then by Theorem \ref{BDRG}, we obtain $(\alpha,\delta)-D-$antimagic labelings of $Q_n, n\equiv 2\pmod4$ for all $D$s which are excluded in Theorem \ref{2mod4}.

The rest of the section will be devoted on searching $D-$magic labelings for the hypercube $Q_n$, where $n\equiv 0\pmod4$. We shall denote by $\boldsymbol{e_i}$ the $i$th standard basis vector in $\F_2^n$. For a vector $\ba=(a_0,\dots,a_{n-1})\in\F_2^n$, we denote by $\zeta(\ba)$ the corresponding nonnegative integer: \[\zeta(\ba)=\sum_{i=0}^{n-1} a_i 2^i,\] where each $a_i\in\F_2$ is regarded as an element of $\Z$. We also denote by $\zeta:\F_2\to\Z$ the embedding obtained by setting $n=1$ in the above definition.

In the $n-$hypercube $Q_n$, we have
\[G_1(\bu)=\{\bu+\boldsymbol{e_i}\mid i\in\{1,\dots,n\}\}\quad(\bu\in\F_2^n).\]
The following definition is essential in finding a closed distance magic labeling of $Q_n$, as can be seen in Lemma \ref{lem:146b}.

\begin{definition}
A subset $A\subseteq\F_2^n$ is said to be \emph{balanced} if
\begin{equation}\label{146a}
|\{\ba\in A\mid a_i=1\}|=\frac{|A|}{2}\quad(\forall i\in\{1,\dots,n\}).
\end{equation}
For a subset $D\subseteq\{0,1,\dots,n\}$, a bijection $f:\F_2^n\to\F_2^n$ is said to be \emph{$D-$neighbor balanced} if $f(\bigcup_{i\in D}G_i(\bu))$ is balanced for every $\bu\in\F_2^n$. If $D=\{1\}$ (resp. $D=\{0,1\}$), then a $D-$neighbor balanced bijection is called a \emph{neighbor balanced} (resp. \emph{closed neighbor balanced}).
\end{definition}
\noindent Note that a subset $A\subseteq\F_2^n$ is balanced if and only if it is an orthogonal array of strength $1$ (see \cite{OA}).

\begin{lemma}\label{lem:146a}
Let $A$ be a balanced subset of $\F_2^n$. Then
\[\sum_{\ba\in A}\zeta(\ba)=\frac{|A|}{2}(2^n-1).\]
\end{lemma}
\begin{proof}
We have
\begin{align*}
\sum_{\ba\in A}\zeta(\ba)&=
\sum_{i=0}^{n-1} \sum_{\ba\in A} a_i 2^i
\nexteq
\sum_{i=0}^{n-1} \frac{|A|}{2} 2^i
\refby{146a}
\nexteq
\frac{|A|}{2}(2^n-1).
\end{align*}
\end{proof}

\begin{lemma}\label{lem:146c}
Let $D\subseteq\{0,1,\dots,n\}$ and let $f:\F_2^n\to \F^n_2$ be a $D-$neighbor balanced bijection. Then the labeling $\zeta\circ f$ is a $D-$magic labeling of $Q_n$.
\end{lemma}
\begin{proof}
Let $\bu\in\F_2^n$, then
\begin{align*}
\sum_{x\in N_D(\bu)}\zeta\circ f(x)&=
\sum_{y\in f(N_D(\bu))}\zeta(y)
\nexteq
\frac{|f(N_D(\bu))|}{2}(2^n-1)
&&\text{(by Lemma~\ref{lem:146a})}
\nexteq
\frac{|N_D(\bu)|}{2}(2^n-1).
\end{align*}
Since $|N_D(\bu)|$ is independent of the choice of $\bu$, we obtain $\zeta\circ f$ is a $D$-magic labeling of $Q_n$.
\end{proof}

\begin{lemma}\label{lem:146b}
Let $f:\F_2^n\to\F_2^n$ be a nonsingular linear transformation. $f$ is closed neighbor-balanced if and only if the matrix representation of $f$ with respect to the standard basis has constant row sum $(n+1)/2$.
\end{lemma}
\begin{proof}
Let $M$ be the matrix representation of $f$ with respect to the standard basis, so that $f(\bu)=M\bu$ for every $\bu\in\F_2^n$.
Then for $i\in\{1,\dots,n\}$ and $\bu\in\F_2^n$, we have
\begin{align*}
&|\{\bv\in f({N}[\bu])\mid f(\bv)_i=1\}|
\\&=|\{j\mid f(\bu+\boldsymbol{e_j})_i=1\}|+\zeta(f(\bu)_i)
\\&=|\{j\mid (M\bu+M\boldsymbol{e_j})_i=1\}|+\zeta((M\bu)_i)
\nexteq
|\{j\mid M_{ij}=(M\bu)_i+1\}|+\zeta((M\bu)_i).
\end{align*}
Thus,
\begin{align*}
&|\{\bv\in f({N}[\bu])\mid f(\bv)_i=1\}|=\frac{n+1}{2}
\quad(\forall \bu\in\F_2^n)
\\&\iff
|\{j\mid M_{ij}=(M\bu)_i+1\}|+\zeta((M\bu)_i)=\frac{n+1}{2}
\quad(\forall \bu\in\F_2^n)
\nexteqv
|\{j\mid M_{ij}=\alpha+1\}|+\zeta(\alpha)=\frac{n+1}{2}
\quad(\forall \alpha\in\F_2)
\nexteqv
|\{j\mid M_{ij}=1\}|=\frac{n+1}{2}.
\end{align*}
Therefore,
\begin{align*}
&\text{$f$ is closed neighbor-balanced}
\\&\iff
|\{\bv\in f({N}[\bu])\mid f(\bv)_i=1\}|
=\frac{n+1}{2}\quad(\forall i\in\{1,\dots,n\},\;\forall
\bu\in\F_2^n)
\nexteqv
|\{j\mid M_{ij}=1\}|=\frac{n+1}{2}\quad(\forall i\in\{1,\dots,n\}).
\end{align*}
\end{proof}

It is known (see \cite[Sect.~III.2]{BI}) that the distance matrices $A_0,A_1,\ldots,A_n$ of $Q_n$ satisfy the following:
\[A_i=v_i(A_1)\quad(i=0,1,\dots,n),\]
where $\{v_i(x)\}_{i=0}^n$ is a sequence of polynomials defined by
\[v_0(x)=1,\quad v_1(x)=x,\]
and
\[xv_i(x)=(i+1)v_{i+1}(x)+(n-i+1)v_{i-1}(x)\quad(i=1,\dots,n-1).\]
It is also known that the eigenvalues of $A_1$ are
\[\theta_j=n-2j\quad(j=0,1,\dots,n).\]
The \emph{Krawtchouk polynomials} $\{K_i(y)\}_{i=0}^n$ are the sequence of polynomials defined by
\[K_i(j)=v_i(\theta_j)\quad(i,j=0,1,\dots,n).\]
Thus,
\begin{equation}\label{K}
(n-2j)K_i(j)=(i+1)K_{i+1}(j)+(n-i+1)K_{i-1}(j)
\end{equation}
It also can be shown by general theory that
\begin{equation}\label{eq:-1}
K_n(j)=(-1)^j\quad(j=0,1,\dots,n).
\end{equation}

Now we are ready to present the necessary and sufficient condition for the existence of a closed distance magic labeling of a hypercube.
\begin{theorem}\label{thm:1}
There exists a closed distance magic labeling of $Q_n$ if and only if $n\equiv1\pmod4$.
\end{theorem}
\begin{proof}
Let $n=4m+1$, where $m$ is a positive integer. Observe that the matrix
\[M=\begin{bmatrix}1&0&\allone_{2m}\\
0&I_{2m}&J_{2m}\\
0&J_{2m}&I_{2m}\end{bmatrix}\]
is nonsingular over $\F_2$ and it has constant row sum $(n+1)/2=2m+1$. By Lemma~\ref{lem:146b}, there exists a closed neighbor-balance bijection $f$ of $\F_2^n$. By Lemma~\ref{lem:146c}, $\zeta\circ f$ is a closed distance magic labeling.

Conversely, suppose that there exists a closed distance magic labeling of $Q_n$. By Lemma \ref{lem:KerB} the kernel of $I+B$ must have a basis $u$ of the form $u_0=1=-u_n$. In particular, $B$ has eigenvalue $-1$, which forces $n$ to be odd. Let $n=2p-1$. Then the eigenvalues of $B$ are $\theta_j=n-2j$, $j=0,1,\dots,n$. Thus $\theta_p=-1$. The normalized eigenvector belonging to the eigenvalue $-1$ is $u=(K_i(p))_{i=0}^n$, and it satisfies $u_0=1$ and $u_n=K_n(p)=(-1)^p$ (by (\ref{eq:-1})). Thus, $p$ must be odd and hence $n\equiv1\pmod{4}$.
\end{proof}

Let $V_j=\kernel(A-\theta_jI)$ denote the eigenspace of $A_1$ corresponding to the eigenvalue $\theta_j$. Then the distance matrix $A_i$ has eigenvalue $K_i(j)$ on $V_j$, that is,
\[V_j\subseteq\kernel(A_i-K_i(j)I)\quad(i,j=0,1,\dots,n).\]
\begin{lemma}\label{lem:A}
\begin{enumerate}
\item
If $n$ is odd then
\[\kernel(A_0+A_1)\subseteq\kernel(A_{2i}+A_{2i+1})\quad(i=0,1,\dots,\frac{n-1}{2}).\]
\item
If $n\equiv1\pmod4$ then
\[\kernel(A_0+A_1)\subseteq\kernel(A_{i}+A_{n-i})\quad(i=0,1,\dots,\frac{n-1}{2}).\]
\end{enumerate}
\end{lemma}
\begin{proof}
(i) Write $n=2p-1$. Then $\theta_p=-1$, so
\begin{equation}\label{Vp}
V_p=\kernel(A_0+A_1)
\end{equation}
The eigenvalue of $A_{2i}+A_{2i+1}$ on $V_p$ is
$K_{2i}(p)+K_{2i+1}(p)$, so it suffices to prove
\[K_{2i}(p)+K_{2i+1}(p)=0\quad(0\leq i<\frac{n}{2}).\]
We prove this by induction on $i$. The case $i=0$ is immediate
from \eqref{Vp}. By the recurrence \eqref{K}, we have
\begin{align*}
-K_{2i+1}(p)&=(2i+2)K_{2i+2}(p)+(2p-2i-1)K_{2i}(p),\\
-K_{2i+2}(p)&=(2i+3)K_{2i+3}(p)+(2p-2i-2)K_{2i+1}(p).
\end{align*}
Thus
\[(2i+3)(K_{2i+3}(p)+K_{2i+2}(p))=(2p-2i-1)(K_{2i}(p)+K_{2i+1}(p)).\]
This completes the inductive step.

(ii) Since $K_n(j)=(-1)^j\quad(j=0,1,\dots,n)$ and $n\equiv1\pmod4$, this implies $K_n(p)=-1$, which in turn implies $V_p\subseteq \kernel(A_0+A_n)$. Since $A_i+A_{n-i}=A_i(A_0+A_n)$, we obtain
\[V_p\subseteq\kernel(A_i+A_{n-i}).\]
The result then follows from \eqref{Vp}.
\end{proof}

\begin{theorem}\label{thm:2}
If $n\equiv1\pmod4$ then there exists a $D$-magic labeling of $Q_n$ whenever $D$ is of the form
\begin{equation}
\bigcup_{i\in I_1}\{2i,2i+1\}\cup\bigcup_{j\in I_2}\{j,n-j\},
\end{equation}
where $I_1,I_2\subseteq\{0,1,\dots,(n-1)/2\}$ and
\[\{2i,2i+1\}\cap\{j,n-j\}=\emptyset\quad(i\in I_1,\;j\in I_2).\]
\end{theorem}
\begin{proof}
By Observation \ref{union}, it suffices to show that there exists a $D$-magic labeling of $Q_n$ for $D$ in
\begin{equation} \label{D}
\{\{2i,2i+1\}\mid 0\leq i< \frac{n}{2}\}\cup\{\{i,n-i\}\mid 0\leq i< \frac{n}{2}\}.
\end{equation}
By Theorem~\ref{thm:1}, there exists a closed distance magic labeling of $Q_n$. By Lemma~\ref{lem:A}, such a labeling is also a $D$-magic labeling for $D$ in \eqref{D}. Indeed, $(A_0+A_1)l\in\langle\allone\rangle$ implies
\begin{align*}
l&\in\kernel(A_0+A_1)+\langle\allone\rangle
\\&\subseteq\kernel(A_{2i}+A_{2i+1})+\langle\allone\rangle.
\end{align*}
Thus $(A_{2i}+A_{2i+1})l\in\langle\allone\rangle$, proving that $l$ is a $\{2i,2i+1\}$-magic labeling. Similarly, we can show that $l$ is a $\{j,n-j\}$-magic labeling.
\end{proof}

\begin{openproblem} \label{op:2}
"Are the sets $D$s in Theorem \ref{thm:2} the only $D$ for which $Q_n$, $n\equiv 2\pmod4$, is $D$-magic?"
\end{openproblem}
Open Problem \ref{op:2} should be true if equality holds in the inclusions in Lemma~\ref{lem:A}.

Additionally, in the next theorem we show that the known distance magic and closed distance magic labelings of $Q_n$ can be utilised to construct $D-$distance antimagic and $D-$closed distance antimagic labelings for $2Q_{n-1}$, the disjoint union of two copies of $Q_{n-1}$.

\begin{theorem}
If $n\equiv 2\pmod{4}$, then $2Q_{n-1}$ is $(\frac{n(2^n+1)}{2}-2^n,1)-$distance antimagic. If $n\equiv 1\pmod{4}$, then $2Q_{n-1}$ is $(\frac{(n+1)(2^n+1)}{2}-2^n,1)-$closed distance antimagic.
\end{theorem}
\begin{proof}
The adjacency matrix $A^{(n)}$ of $Q_n$ is of the form
\[\begin{pmatrix}A^{(n-1)}&I\\I&A^{(n-1)}\end{pmatrix},\]
where $A^{(n-1)}$ is the adjacency matrix of $Q_{n-1}$.

Suppose first $n\equiv 2\pmod4$. Then by Theorem \ref{QnDM}, $Q_n$ has a distance magic labeling $l$. Partitioning the vector $l$ into two equal parts $l^{(1)}$ and $l^{(2)}$,
we find
\[\begin{pmatrix}A^{(n-1)}&0\\0&A^{(n-1)}\end{pmatrix}l
\in\begin{pmatrix}l^{(2)}\\ l^{(1)}\end{pmatrix}+\langle\allone\rangle.\]
This implies that $l$ is a
$(\frac{n(2^n+1)}{2}-2^n,1)$-distance antimagic labeling.
Therefore, $2Q_{n-1}$ is $(\frac{n(2^n+1)}{2}-2^n,1)-$distance antimagic.

Next suppose $n\equiv 1\pmod{4}$. Then Theorem~\ref{thm:1} implies that there exists a closed distance magic labeling $l$ of $Q_n$. If we regard $l$ as a labeling of $2Q_{n-1}$, then it is easy to see that $l$ is a $(\frac{(n+1)(2^n+1)}{2}-2^n,1)-$closed distance antimagic labeling, by a similar argument as the previous case.
\end{proof}




\end{document}